\documentclass[a4paper,11pt]{article}
\usepackage{amsmath}
\usepackage{amsfonts}
\usepackage{amssymb}
\usepackage{amsthm}
\usepackage{graphicx}

\date{} 
\begin{document} 
\centerline {\bf {On Generalization of $\displaystyle\delta$-Primary Elements in Multiplicative Lattices} }
\centerline{} 

\centerline{\bf {A.~V.~Bingi}}
\centerline{Department of Mathematics}
\centerline{St.~Xavier's College(autonomous), Mumbai-400001, India}
\centerline{$email: ashok.bingi@xaviers.edu$}

\newtheorem{thm}{Theorem}[section]
 \newtheorem{c1}[thm]{Corollary}
 \newtheorem{l1}[thm]{Lemma}
 \newtheorem{prop}[thm]{Proposition}
 \newtheorem{d1}[thm]{Definition}
\newtheorem{rem}[thm]{Remark}
 \newtheorem{e1}[thm]{Example}
\begin{abstract}
  In this paper, we introduce $\phi$-$\delta$-primary elements in a compactly generated multiplicative lattice $L$ and obtain its characterizations. We prove many of its properties and investigate the relations between these structures. By a counter example, it is shown that a $\phi$-$\delta$-primary element of $L$ need not be $\delta$-primary and found conditions under which a $\phi$-$\delta$-primary element of $L$ is $\delta$-primary.
\end{abstract}

\paragraph*{}
\section{Introduction}
\paragraph*{}  
Prime ideals play a central role in commutative ring theory. In the literature, we find that there are several ways to generalize the notions of a prime ideal and a primary ideal of a commutative ring $R$ with unity. A prime ideal $P$ of $R$ is an ideal with the property that for all $a,\ b\in R$, $ab\in P$ implies either $a\in P$ or $b\in P$. We can either restrict or enlarge where $a$ and/or $b$ lie or restrict or enlarge where $ab$ lies. Same can be thought for primary ideals too. As a generalization of prime ideals of $R$, $\phi$-prime ideals were introduced in $\cite{AB2}$ and $\cite{EN}$ while as a generalization of primary ideals of $R$, $\phi$-primary ideals were introduced in $\cite{D}$. In an attempt to unify the prime and primary ideals of $R$ under one frame, $\delta$-primary ideals of $R$ were introduced in $\cite{Z}$. Further, the concept of $\delta$-primary ideals of $R$ was generalized by introducing the notion of $\phi$-$\delta$-primary ideals of $R$ in $\cite{J1}$.

As an extension of these concepts of a commutative ring $R$ to a multiplicative lattice $L$, C.~S.~Manjarekar and A.~V.~Bingi introduced $\delta$-primary elements of $L$ in \cite{MB1} and introduced $\phi$-prime, $\phi$-primary elements of $L$ in \cite{MB}. In this paper, we introduce and study, $\phi$-$\delta$-primary elements of $L$ as a generalization of $\delta$-primary elements of $L$ and unify $\phi$-prime and $\phi$-primary elements of $L$ under one frame.

A multiplicative lattice $L$ is a complete lattice provided with commutative, associative and join distributive multiplication in which the largest element $1$ acts as a multiplicative identity. An element $e\in L$ is called meet principal if $a\wedge be=((a:e)\wedge b)e$ for all $a,\ b\in L$. An element $e\in L$ is called join principal if $(ae\vee b):e=(b:e)\vee a$ for all $a,\ b\in L$. An element $e\in L$ is called principal if $e$ is both meet principal and join principal. A multiplicative lattice $L$ is said to be principally generated(PG) if every element of $L$ is a join of principal elements of $L$. An element $a\in L$ is called compact if for  $X\subseteq L$, $a\leqslant \vee X$  implies the existence of a finite number of elements $a_1,a_2,\cdot\cdot\cdot,a_n$ in $X$ such that $a\leqslant a_1\vee a_2\vee\cdot\cdot\cdot\vee a_n$. The set of compact elements of $L$ will be denoted by $L_\ast$. If each element of $L$ is a join of compact elements of $L$ then $L$ is called a compactly generated lattice or simply a CG-lattice. 
       
An element $a\in L$ is said to be proper if $a<1$. The radical of $a\in L$ is denoted by $\sqrt{a}$ and is defined as $\vee \{ x \in L_\ast \mid x^{n} \leqslant a$, for  some  $n\in Z_+\}$. A proper element $m\in L$ is said to be maximal if for every element $x\in L$ such that $m<x\leqslant 1$ implies $x=1$. A proper element $p\in L$ is called a prime element if $ab\leqslant p$ implies $a\leqslant p$ or $b\leqslant p$ where $a,b\in L$ and is called a primary element if $ab\leqslant p$ implies $a\leqslant p$ or $b\leqslant \sqrt{p}$ where $a,b\in L_\ast$. For $a,b\in L $, $(a:b)= \vee \{x \in L \mid xb \leqslant a\} $. A multiplicative lattice is called as a Noether lattice if it is modular, principally generated and satisfies ascending chain condition.  An element $a\in L$ is called a zero divisor if $ab=0$ for some $0\neq b\in L$ and is called idempotent if $a=a^2$. A multiplicative lattice is said to be a domain if it is without zero divisors and is said to be quasi-local if it contains a unique maximal element. A quasi-local multiplicative lattice $L$ with maximal element $m$ is denoted by $(L,\ m)$. A Noether lattice $L$ is local if it contains precisely one maximal prime. In a Noether lattice $L$, an element $a\in L$ is said to satisfy  restricted cancellation law if for all $b,\ c\in L$, $ab=ac\neq0$ implies $b=c$ (see $\cite{W}$). According to \cite{MB1}, An expansion function on $L$ is a function $\delta:L \longrightarrow L$ which satisfies the following two conditions: \textcircled{1}. $a\leqslant\delta(a)$ for all $a\in L$,
\textcircled{2}. $a\leqslant b$ implies $\delta(a)\leqslant\delta(b)$ for all $a,\ b\in L$ and a proper element $p\in L$ is called $\delta$-primary if for all $a,\ b\in L$, $ab\leqslant p$ implies either $a\leqslant p$ or $b\leqslant\delta(p)$. According to \cite{MB},  a proper element $p\in L$ is said to be $\phi$-prime if for all $a,\ b\in L$, $ab\leqslant p$ and $ab\nleqslant\phi(p)$ implies either $a\leqslant p$ or $b\leqslant p$ and a proper element $p\in L$ is said to be $\phi$-primary if for all $a,\ b\in L$, $ab\leqslant p$ and $ab\nleqslant\phi(p)$ implies either $a\leqslant p$ or $b\leqslant\sqrt p$ where $\phi:L \longrightarrow L$ is a function on $L$. The reader is referred to \cite{AAJ} and \cite{D3} for general background and terminology in multiplicative lattices.      
             
         This paper is motivated by \cite{J1}. In this paper, we define a $\phi$-$\delta$-primary element in $L$ and obtain their characterizations. Various $\phi_\alpha$-$\delta$-primary elements of $L$ are introduced and relations among them are obtained. By counter examples, it is shown that a $\phi$-$\delta$-primary element of $L$ need not be $\phi$-prime, a $\phi$-$\delta$-primary element of $L$ need not be prime and a $\phi$-$\delta$-primary element of $L$ need not be $\delta$-primary. In 7 different ways, we have proved that a $\phi$-$\delta$-primary element of $L$ is $\delta$-primary under certain conditions. We define a $2$-potent $\delta$-primary element of $L$ and a $n$-potent $\delta$-primary element of $L$. We investigate some properties of $\phi$-$\delta$-primary elements of $L$ with respect to lattice homomorphism and global property. Finally, we show that every idempotent element of $L$ is $\phi_2$-$\delta$-primary but converse need not be true. Throughout this paper, \textcircled{1}. $L$ denotes a compactly generated multiplicative lattice with $1$ compact in which every finite product of compact elements is compact, \textcircled{2}. $\delta$ denotes an expansion function on $L$ and \textcircled{3}. $\phi$ denotes a function defined on $L$.
         
       \section{$\phi$-$\delta$-primary elements of $L$}
       
       \paragraph*{} 
       We begin with introducing the notion of  $\phi$-$\delta$-primary elements of $L$ which is the      generalization of the concept of $\delta$-primary elements of $L$. 
        
              \begin{d1}\label{D-C2} 
               Given an expansion function $\delta:L\longrightarrow L$ and a function $\phi:L\longrightarrow L$, a proper element $p\in L$ is said to be $\phi$-$\delta$-primary if for all $a,\ b\in L$, $ab\leqslant p$ and $ab\nleqslant\phi(p)$ implies either $a\leqslant p$ or $b\leqslant \delta(p)$.
              \end{d1}
              
              For the special functions $\phi_\alpha:L\longrightarrow L$, the $\phi_\alpha$-$\delta$-primary elements of $L$ are defined by following settings in the definition \ref{D-C2} of a $\phi$-$\delta$-primary element of $L$. For any proper element $p\in L$ in the definition \ref{D-C2}, in place of $\phi(p)$, set
                     \begin{itemize}
                     \item $\phi_0(p)=0$. Then $p\in L$ is called a {\bf weakly $\delta$-primary} element.
                     \item $\phi_2(p)=p^2$. Then $p\in L$ is called a {\bf $2$-almost $\delta$-primary} element or  a {\bf $\phi_2$-$\delta$-primary} element or simply an {\bf almost primary} element.
                     \item $\phi_n(p)=p^n\ (n>2)$. Then $p\in L$ is called an {\bf $n$-almost $\delta$-primary} element or a {\bf $\phi_n$-$\delta$-primary} element $(n>2)$.
                     \item $\phi_\omega(p)=\bigwedge_{i=1}^{\infty}p^n$. Then $p\in L$ is called a {\bf $\omega$-$\delta$-primary} element or {\bf $\phi_\omega$-$\delta$-primary} element.
                     \end{itemize} 
                                          
  Since for an element $a\in L$ with $a\leqslant q$ but $a\nleqslant \phi(q)$ implies that $a \nleqslant q\wedge \phi(q)$, there is no loss generality in assuming that $\phi(q) \leqslant q$. We henceforth make this assumption. 
              
     \begin{d1}
     Given any two functions $\gamma_1, \gamma_2 : L\longrightarrow L$, we define $\gamma_1\leqslant \gamma_2$ if $\gamma_1(a)\leqslant \gamma_2(a)$ for each $a\in L$.
     \end{d1} 
     
     Clearly, we have the following order:
     
     \centerline{ $\phi_0 \leqslant \phi_\omega \leqslant \cdots \leqslant \phi_{n+1} \leqslant \phi_n \leqslant \cdots \leqslant \phi_2 \leqslant \phi_1$}
     
 \medskip
 
  Further as $\phi(p)\leqslant p$ and $p\leqslant \delta(p)$ for each $p\in L$, the relation between the functions $\delta$ and $\phi$ is $\phi\leqslant \delta$.
  
  \medskip  
      
  According to \cite{MB1}, $\delta_0$ is an expansion function on $L$ defined as $\delta_0(p)=p$ for each $p\in L$ and $\delta_1$ is an expansion function on $L$ defined as $\delta_1(p)=\sqrt{p}$ for each $p\in L$. Further, note that by Theorem 2.2. in \cite{MB1}, a proper element $p\in L$ is $\delta_0$-primary if and only if it is prime and by Theorem 2.3. in \cite{MB1}, a proper element $p\in L$ is $\delta_1$-primary if and only if it is primary.\\
 
 The following 2 results relate $\phi$-prime and $\phi$-primary elements of $L$ with some $\phi$-$\delta$-primary elements of $L$.
 
 \begin{thm}
 A proper element $p\in L$ is $\phi$-$\delta_0$-primary if and
  only if $p$ is $\phi$-prime.
 \end{thm}
 \begin{proof}
 The proof is obvious.
 \end{proof}
 
  \begin{thm}
  A proper element $p\in L$ is $\phi$-$\delta_1$-primary if and
   only if $p$ is $\phi$-primary.
  \end{thm}
  \begin{proof}
  The proof is obvious.
  \end{proof}

 \begin{thm}\label{T-C21}
 Let $\delta,\ \gamma :L \longrightarrow L$ be expansion functions on $L$ such that $\delta \leqslant\gamma$. Then every $\phi$-$\delta$-primary element of $L$ is $\phi$-$\gamma$-primary. In particular, a $\phi$-prime element of $L$ is $\phi$-$\delta$-primary for every expansion function $\delta$ on $L$.
\end{thm}
         \begin{proof}Let a proper element $p\in L$ be $\phi$-$\delta$-primary. Suppose $ab\leqslant p$ and $ab\nleqslant \phi(p)$ for $a,\ b\in L$. Then either $a\leqslant p$ or $b\leqslant \delta(p)\leqslant\gamma(p)$ and so $p$ is $\phi$-$\gamma$-primary. Next, for any expansion function $\delta$ on $L$, we have $\delta_0\leqslant \delta$. So a $\phi$-$\delta_0$-primary element of $L$ is $\phi$-$\delta$-primary and we are done since a  $\phi$-prime element of $L$ is $\phi$-$\delta_0$-primary.
         \end{proof}
         
         \begin{c1}
         A prime element of $L$ is $\phi$-$\delta$-primary for every expansion function $\delta$ on $L$.
         \end{c1}
         \begin{proof}
         The proof follows by using Theorem \ref{T-C21} to the fact that every prime element of $L$ is $\phi$-prime.
         \end{proof}

         The following example shows that (by taking $\phi$ as $\phi_2$ and $\delta$ as $\delta_1$ for convenience) 
         
         \textcircled{1}. a $\phi$-$\delta$-primary element of $L$ need not be $\phi$-prime,
          
         \textcircled{2}. a $\phi$-$\delta$-primary element of $L$ need not be prime.
         
         \begin{e1}\label{E-C21}
         Consider the lattice $L$ of ideals of the ring $R=<Z_{24}\ , \ +\ ,\ \cdot>$. Then the only ideals of $R$ are the principal ideals (0),(2),(3),(4),(6),(8),(12),(1). Clearly, $L=\{$(0),(2),(3),(4),(6),(8),(12),(1)$\}$ is a compactly generated multiplicative lattice. Its lattice structure and multiplication table is as shown in Figure 2.1. It is easy to see that the element $(4)\in L$ is $\phi_2$-$\delta_1$-primary while $(4)$ is not  $\phi_2$-prime because though $(2)\cdot (6)\subseteq (4)$, $(2)\cdot (6)\nsubseteq (4)^2$ but $(2)\nsubseteq (4)$ and $(6)\nsubseteq (4)$. Also, $(4)$ is not prime.\\
         
         \centerline{\includegraphics[width = 125mm, height = 75mm]{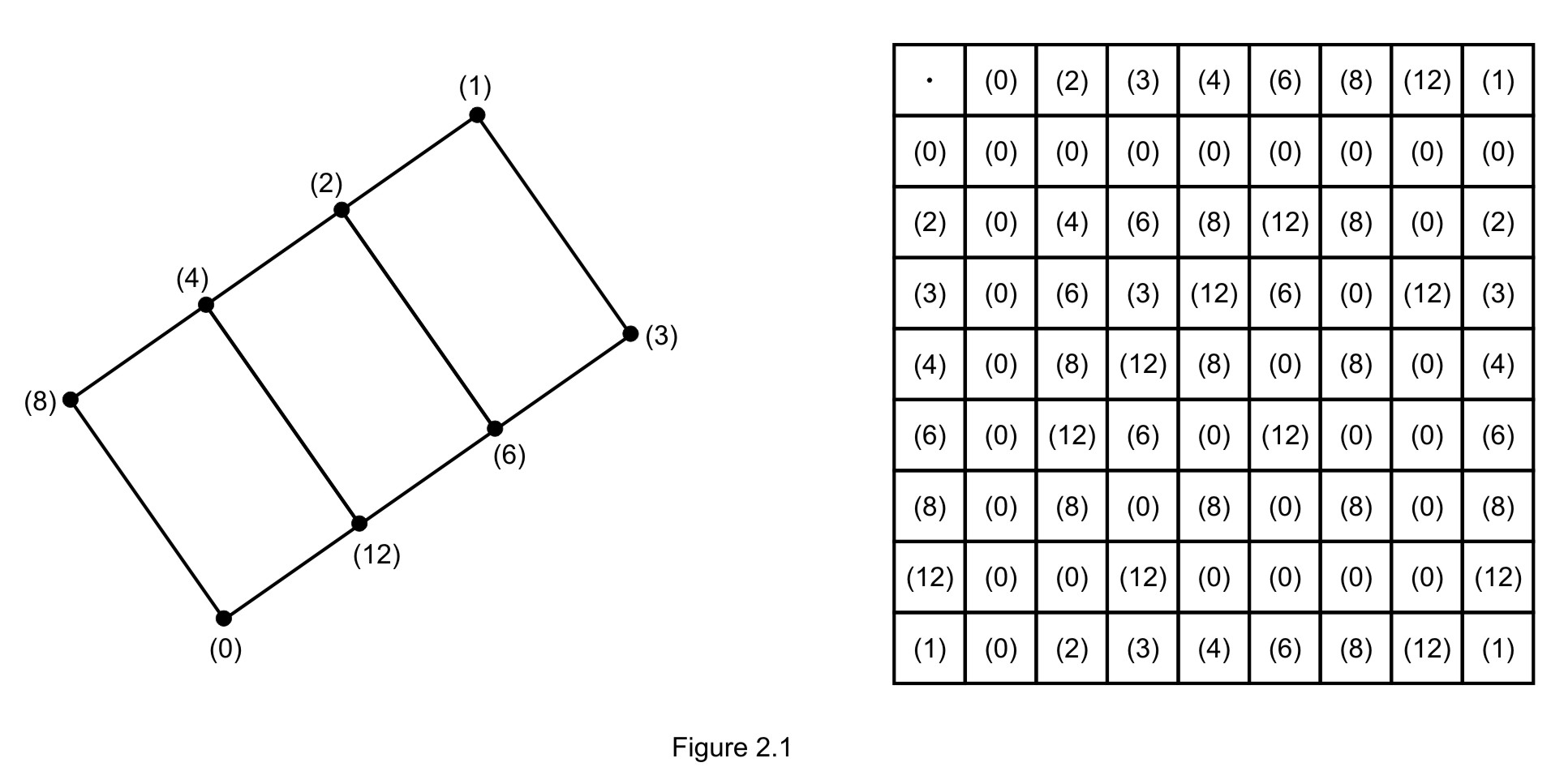}}
        \end{e1}
                 
   Now before obtaining the characterizations of a $\phi$-$\delta$-primary element of $L$, we state the following essential lemma which is outcome of Lemma 2.3.13 from \cite{S0}.
                    
                    \begin{l1}\label{L}
                    Let $a_1,\ a_2\in L$. Suppose $b\in L$ satisfies the following property:
                    
                    ($\ast$). If $h\in L_\ast$ with $h\leqslant b$ then either $h\leqslant a_1$ or $h\leqslant a_2$.
                    
                    Then either $b\leqslant a_1$ or $b\leqslant a_2$.
                    \end{l1}                             
             
         \begin{thm}\label{T-C26}
         Let $q$ be a proper element of $L$. Then the following statements are equivalent:
         
         \textcircled{1}. $q$ is $\phi$-$\delta$-primary. 
         
         \textcircled{2}. for every $a\in L$ such that $a\nleqslant \delta(q)$, either $(q:a)=q$ or $(q:a)=(\phi(q):a)$. 
         
         \textcircled{3}. for every $r,\ s\in L_\ast$, $rs\leqslant q$ and $rs\nleqslant \phi(q)$ implies either $s\leqslant q$ or $r\leqslant \delta(q)$. 
         \end{thm}
         \begin{proof}
         \textcircled{1}$\Longrightarrow$\textcircled{2}. Suppose \textcircled{1} holds. Let $h\in L_\ast$ be such that $h\leqslant (q:a)$ and $a\nleqslant \delta(q)$. Then $ah\leqslant q$. If $ah\leqslant \phi(q)$ then $h\leqslant (\phi(q):a)$. If $ah\nleqslant \phi(q)$ then since $q$ is $\phi$-$\delta$-primary and $a\nleqslant \delta(q)$, it follows that $h\leqslant q$. Hence by Lemma $\ref{L}$, either $(q:a)\leqslant (\phi(q):a)$ or $(q:a)\leqslant q$. Consequently, either $(q:a)=(\phi(q):a)$ or $(q:a)=q$.
         
         \textcircled{2}$\Longrightarrow$\textcircled{3}. Suppose \textcircled{2} holds. Let $rs\leqslant q$, $rs\nleqslant \phi(q)$ and $r\nleqslant \delta(q)$ for  $r,\ s\in L_\ast$. Then by \textcircled{2}, either $(q:r)=(\phi(q):r)$ or $(q:r)=q$. If $(q:r)=(\phi(q):r)$ then as $s\leqslant (q:r)$, it follows that $s\leqslant (\phi(q):r)$ which contradicts $rs\nleqslant \phi(q)$ and so we must have $(q:r)=q$. Therefore $s\leqslant (q:r)$ gives $s\leqslant q$.
         
         \textcircled{3}$\Longrightarrow$\textcircled{1}. Suppose \textcircled{3} holds. Let $ab\leqslant q$, $ab\nleqslant \phi(q)$ and $a\nleqslant\delta(q)$ for $a,\ b\in L$. Then as $L$ is compactly generated, there exist $x,\ x',\ y'\in L_\ast$ such that $x\leqslant a,\ x'\leqslant a,\ y'\leqslant b,\ x\nleqslant \delta(q)$ and $x'y'\nleqslant \phi(q)$. Let $y\leqslant b$ be any compact element of $L$. Then $(x\vee x'),\ (y\vee y')\in L_\ast$ such that $(x\vee x')(y\vee y')\leqslant q,\ (x\vee x')(y\vee y')\nleqslant \phi(q)$ and $(x\vee x')\nleqslant\delta(q)$. So by \textcircled{3}, it follows that $(y\vee y')\leqslant q$ which implies $b\leqslant q$ and therefore $q$ is $\phi$-$\delta$-primary.
         \end{proof}   
         
\begin{thm}\label{T-C27} 
 A proper element $q\in L$ is $\phi$-$\delta$-primary if and only if for every $a\in L$ such that $a\nleqslant q$ either $(q:a)\leqslant \delta(q)$ or $(q:a)=(\phi(q):a)$.
 \end{thm}
 \begin{proof} Assume that a proper element $q\in L$ is $\phi$-$\delta$-primary. Let $h\in L_\ast$ be such that $h\leqslant (q:a)$ and $a\nleqslant q$. Then $ah\leqslant q$. If $ah\leqslant \phi(q)$ then $h\leqslant (\phi(q):a)$. If $ah\nleqslant  \phi(q)$ then since $q$ is $\phi$-$\delta$-primary  and $a\nleqslant q$, it follows that $h\leqslant \delta(q)$. Hence by Lemma $\ref{L}$, either $(q:a)\leqslant (\phi(q):a)$ or $(q:a)\leqslant\delta(q)$. But as $(\phi(q):a)\leqslant (q:a)$ we have either $(q:a)\leqslant \delta(q)$ or $(q:a)=(\phi(q):a)$. Conversely, assume that for every $a\in L$ such that $a\nleqslant q$, either $(q:a)\leqslant \delta(q)$ or $(q:a)=(\phi(q):a)$. Let $rs\leqslant q$, $rs\nleqslant \phi(q)$ and $r\nleqslant q$ for $r,\ s\in L$. Then either $(q:r)=(\phi(q):r)$ or $(q:r)\leqslant \delta(q)$. If  $(q:r)=(\phi(q):r)$ then as $s\leqslant (q:r)$, it follows that $s\leqslant (\phi(q):r)$ which contradicts $rs\nleqslant \phi(q)$ and so we must have $(q:r)\leqslant\delta(q)$. Therefore $s\leqslant (q:r)$ gives $s\leqslant\delta(q)$. Hence $q$ is $\phi$-$\delta$-primary.
 \end{proof}
 
 \begin{thm}
 Let $(L,\ m)$ be a quasi-local Noether lattice. If a proper element $p\in L$ is such that $p^2=m^2\leqslant p\leqslant m$ then $p$ is $\phi_2$-$\delta_1$-primary.
 \end{thm}
 \begin{proof}
 Let $xy\leqslant p$ and $xy\nleqslant \phi_2(p)$ for $x,\ y\in L$. If $x\nleqslant m$ then $x=1$. So $xy\leqslant p$ gives $y\leqslant p$. Similarly, $y\nleqslant m$ gives $x\leqslant p$.  Now if $x\leqslant m$ then $x^2\leqslant m^2=p^2\leqslant p$ and hence $x\leqslant\delta_1(p)$. Similarly, $y\leqslant m$ gives $y\leqslant\delta_1(p)$. Hence in any case, $p$ is $\phi_2$-$\delta_1$ primary.
 \end{proof}                  
             
      To obtain the relation among $\phi_\alpha$-$\delta$-primary elements of $L$, we prove the following lemma.
             
     \begin{l1}\label{L-C1}
            Let $\gamma_1,\ \gamma_2 :L \longrightarrow L$ be functions such that $\gamma_1\leqslant\gamma_2$ and $\delta$ be an expansion function on $L$. Then every proper $\gamma_1$-$\delta$-primary element of $L$ is $\gamma_2$-$\delta$-primary.
            \end{l1}
            \begin{proof} Let a proper element $p\in L$ be $\gamma_1$-$\delta$-primary. Suppose $ab\leqslant p$ and $ab\nleqslant \gamma_2(p)$ for $a,\ b\in L$. Then as $\gamma_1\leqslant\gamma_2$, we have $ab\leqslant p$ and $ab\nleqslant \gamma_1(p)$. Since $p$ is $\gamma_1$-$\delta$-primary, it follows that either $a\leqslant p$ or $b\leqslant\delta(p)$ and hence $p$ is $\gamma_2$-$\delta$-primary.
            \end{proof} 
            
             \begin{thm}\label{T-C22}
             For a proper element of $L$, consider the following statements:
             \begin{enumerate}
             \item[(a)] $p$ is a $\delta$-primary element of $L$. 
             \item[(b)] $p$ is a $\phi_0$-$\delta$-primary element of $L$.
             \item[(c)] $p$ is a $\phi_\omega$-$\delta$-primary element of $L$.
             \item[(d)] $p$ is a $\phi_{(n+1)}$-$\delta$-primary element of $L$.
             \item[(e)] $p$ is a $\phi_n$-$\delta$-primary element of $L$ where $n\geqslant 2$.
             \item[(f)] $p$ is a $\phi_2$-$\delta$-primary element of $L$.
             \end{enumerate}
             Then $(a) \Longrightarrow (b) \Longrightarrow (c) \Longrightarrow (d) \Longrightarrow (e) \Longrightarrow (f)$.
           \end{thm}
          \begin{proof}
          Obviously, every $\delta$-primary element of $L$ is weakly $\delta$-primary and hence $(a) \Longrightarrow (b)$. The remaining implications follow by using Lemma \ref{L-C1} to the fact that $\phi_0 \leqslant \phi_\omega \leqslant \cdots \leqslant \phi_{n+1} \leqslant \phi_n \leqslant \cdots \leqslant \phi_2$
          \end{proof}        
              
     \begin{c1}\label{L-C21}
       Let $p\in L$ be a proper element. Then $p$ is $\phi_\omega$-$\delta$-primary if and only if $p$ is $\phi_n$-$\delta$-primary for every $n\geqslant 2$.
       \end{c1}
       \begin{proof}
       Assume that $p\in L$ is $\phi_n$-$\delta$-primary  for every $n\geqslant 2$. Let $ab\leqslant p$ and $ab\nleqslant\bigwedge_{n=1}^{\infty}p^n$ for $a,\ b\in L$. Then $ab\leqslant p$ and $ab\nleqslant p^n$ for some $n\geqslant 2$. Since $p$ is $\phi_n$-$\delta$-primary, we have either $a\leqslant p$ or $b\leqslant\delta(p)$ and hence $p$ is $\phi_\omega$-$\delta$-primary. The converse follows from Theorem $\ref{T-C22}$.
       \end{proof}
       
   Now we show that under a certain condition, a $\phi_n$-$\delta$-primary element of $L$~ $(n\geqslant 2)$ is $\delta$-primary.
     
     \begin{thm}\label{T-C74}
     Let $L$ be a local Noetherian domain. A proper element $p\in L$ is $\phi_n$-$\delta$-primary for every $n\geqslant 2$ if and only if $p$ is $\delta$-primary.
     \end{thm}
     \begin{proof}
     Assume that a proper element $p\in L$ is $\phi_n$-$\delta$-primary for every $n\geqslant 2$. Let $ab\leqslant p$ for $a,\ b\in L$. If $ab\nleqslant \phi_n(p)$ for $n\geqslant 2$ then as $p\in L$ is $\phi_n$-$\delta$-primary, we have $a\leqslant p$ or $b\leqslant \delta(q)$. If $ab\leqslant \phi_n(p)=p^n$ for all $n\geqslant 1$ then as $L$ is local Noetherian, by Corollary 3.3 of $\cite{D3}$, it follows that $ab\leqslant\bigwedge_{n=1}^{\infty}p^n=0$ and so $ab=0$. Since $L$ is domain, we have either $a=0$ or $b=0$ which implies either $a\leqslant p$ or $b\leqslant \delta(q)$ and hence $p$ is $\delta$-primary. Converse follows from Theorem \ref{T-C22}.
     \end{proof}
     
     \begin{c1}
     Let $L$ be a local Noetherian domain. A proper element $p\in L$ is $\phi_\omega$-$\delta$-primary if and only if $p$ is $\delta$-primary.
     \end{c1}
     \begin{proof}
     The proof follows from Theorem $\ref{T-C74}$ and Corollary $\ref{L-C21}$.
     \end{proof}
          
 Clearly, every $\delta$-primary element of $L$ is $\phi$-$\delta$-primary. The following example shows that its converse is not true (by taking $\phi$ as $\phi_2$ and $\delta$ as $\delta_1$ for convenience).
            
            \begin{e1}\label{E-C22}
            Consider the lattice $L$ of ideals of the ring $R=<Z_{30}\ , \ +\ ,\ \cdot>$. Then the only ideals of $R$ are the principal ideals (0),(2),(3),(5),(6),(10),(15),(1). Clearly $L=\{$(0),(2),(3),(5),(6),(10),(15),(1)$\}$ is a compactly generated multiplicative lattice. Its lattice structure and multiplication table is as shown in Figure 2.2. It is easy to see that the element $(6)\in L$ is $\phi_2$-$\delta_1$-primary but not $\delta_1$-primary.
            \centerline{\includegraphics[width = 125mm, height = 75mm]{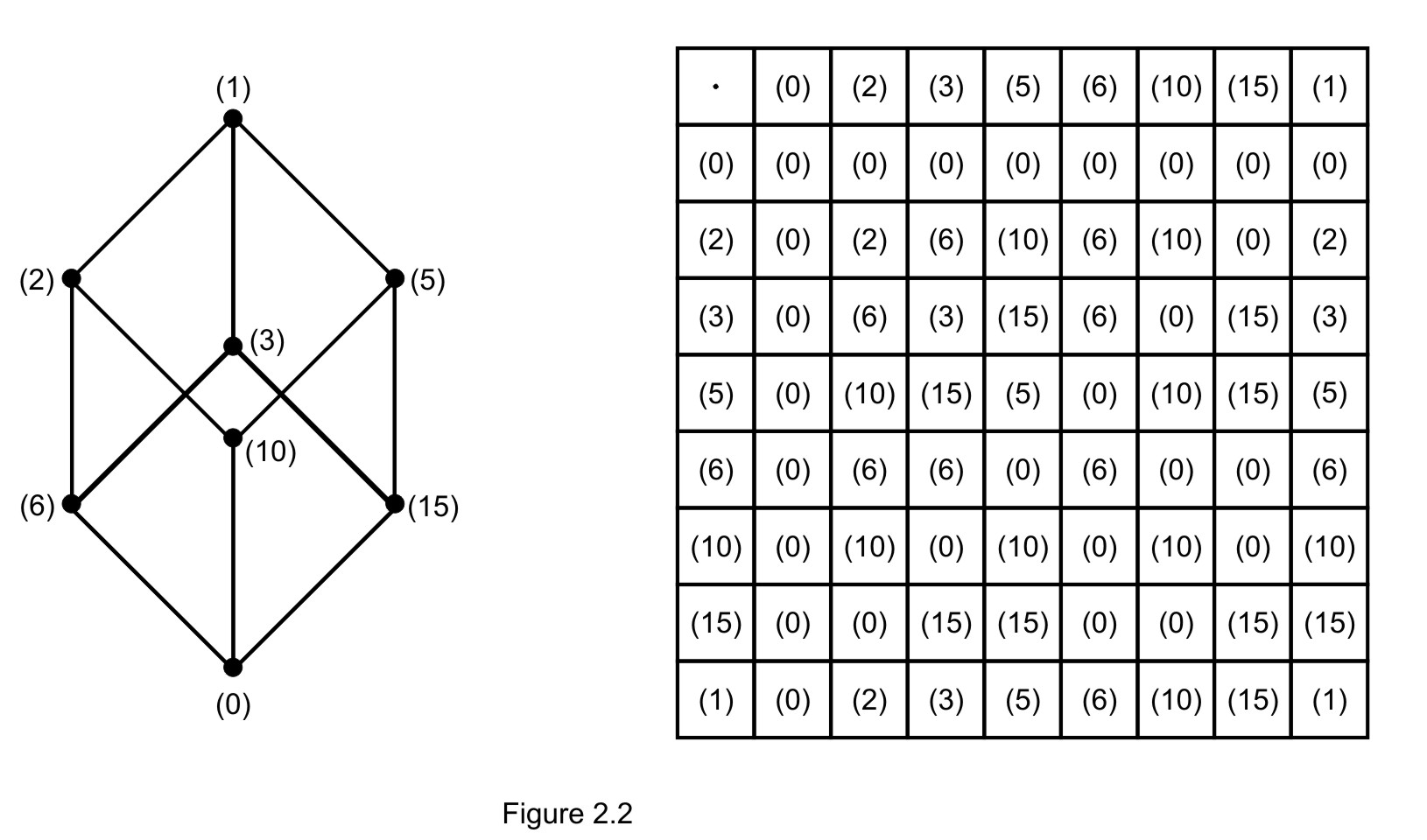}}
            \end{e1} 
   
   In the following successive seven theorems, we show conditions under which a $\phi$-$\delta$-primary element of $L$ is $\delta$-primary.
   
   \begin{thm}\label{T-C72}
   Let $L$ be a Noether lattice. Let $0\neq q\in L$ be a non-nilpotent proper element satisfying the restricted cancellation law. Then $q$ is $\phi$-$\delta$-primary for some $\phi\leqslant\phi_2$ if and only if $q$ is $\delta$-primary.
   \end{thm} 
   \begin{proof}
   Assume that $q\in L$ is a $\delta$-primary element. Then obviously, $q$ is $\phi$-$\delta$-primary for every $\phi$ and hence for some $\phi\leqslant \phi_2$. Conversely, let $q\in L$ be $\phi$-$\delta$-primary for some $\phi\leqslant\phi_2$. Then by Lemma \ref{L-C1}, $q\in L$ is $\phi_2$-$\delta$-primary (almost $\delta$-primary). Let $xy\leqslant q$ for $x,\ y\in L$. If $xy\nleqslant \phi_2(q)$ then as $q$ is $\phi_2$-$\delta$-primary, we have either $x\leqslant q$ or $y\leqslant\delta(q)$. If $xy\leqslant \phi_2(q)=q^2$,  consider $(x\vee q)y=xy\vee qy\leqslant q$. If $(x\vee q)y\nleqslant \phi_2(q)$ then as $q$ is $\phi_2$-$\delta$-primary, we have either $x\leqslant (x\vee q)\leqslant q$ or  $y\leqslant\delta(q)$. So assume that $(x\vee q)y\leqslant \phi_2(q)$. Then $qy\leqslant q^2\neq 0$ which implies $y\leqslant q\leqslant \delta(q)$ by Lemma 1.11 of $\cite{W}$. Hence $q$ is $\delta$-primary.
   \end{proof}
   
   \begin{c1}
   Every non-zero and non-nilpotent $\phi_2$-$\delta$-primary element of a Noether lattice $L$ satisfying the restricted cancellation law is $\delta$-primary.
   \end{c1}
   \begin{proof}
   The proof follows from proof of the Theorem \ref{T-C72}.
   \end{proof}
   
   The following result is general form of Theorem \ref{T-C72}.
   
   \begin{thm}\label{T-C79}
      Let $L$ be a Noether lattice. Let $0\neq q\in L$ be a non-nilpotent proper element satisfying the restricted cancellation law. Then $q$ is $\phi$-$\delta$-primary for some $\phi\leqslant\phi_n$ and for all $n\geqslant 2$ if and only if $q$ is $\delta$-primary.
      \end{thm} 
      \begin{proof}
      Assume that $q\in L$ is a $\delta$-primary element. Then obviously, $q$ is $\phi$-$\delta$-primary for every $\phi$ and hence for some $\phi\leqslant \phi_n$, for all $n\geqslant 2$. Conversely, let $q\in L$ be $\phi$-$\delta$-primary for some $\phi\leqslant\phi_n$ and for all $n\geqslant 2$. Then by Lemma \ref{L-C1}, $q\in L$ is $\phi_n$-$\delta$-primary ($n$-almost $\delta$-primary) and for all $n\geqslant 2$. Let $xy\leqslant q$ for $x,\ y\in L$. If $xy\nleqslant \phi_n(q)$ for some $n\geqslant 2$ then as $q$ is $\phi_n$-$\delta$-primary, we have either $x\leqslant q$ or $y\leqslant\delta(q)$ and we are done. So let $xy\leqslant \phi_n(q)=q^n$ for all $n\geqslant 2$.  Consider $(x\vee q)y=xy\vee qy\leqslant q$. If $(x\vee q)y\nleqslant \phi_n(q)$ then as $q$ is $\phi_n$-$\delta$-primary, we have either $x\leqslant (x\vee q)\leqslant q$ or  $y\leqslant\delta(q)$. So assume that $(x\vee q)y\leqslant \phi_n(q)$. Then $qy\leqslant q^n\leqslant q^2\neq 0$ as $n\geqslant 2$. This implies $y\leqslant q\leqslant \delta(q)$ by Lemma 1.11 of $\cite{W}$. Hence $q$ is $\delta$-primary.
      \end{proof}
      
      \begin{c1}
      Every non-zero and non-nilpotent $\phi_n$-$\delta$-primary element $(\forall$  $n\geqslant 2)$ of a Noether lattice $L$  satisfying the restricted cancellation law is $\delta$-primary.
      \end{c1}
      \begin{proof}
      The proof follows from proof of the Theorem $\ref{T-C79}$.
      \end{proof}
      
  \begin{d1} 
   A proper element $p\in L$ is said to be {\bf 2-potent $\delta$-primary} if for all $a,\ b\in L$, $ab\leqslant p^2$ implies either $a\leqslant p$ or $b\leqslant \delta(p)$.
   \end{d1} 
   
   Obviously, every $2$-potent $\delta_0$-primary element of $L$ is $2$-potent prime and vice versa. Also, every $2$-potent $\delta_0$-primary element of $L$ is $2$-potent $\delta$-primary.
   
   \begin{thm}\label{T-C80}
   Let a proper element $q\in L$ be 2-potent $\delta$-primary. Then $q$ is $\phi$-$\delta$-primary for some $\phi\leqslant\phi_2$ if and only if $q$ is $\delta$-primary.
      \end{thm} 
      \begin{proof}
      Assume that $q\in L$ is a $\delta$-primary element. Then obviously, $q$ is $\phi$-$\delta$-primary for every $\phi$ and hence for some $\phi\leqslant \phi_2$. Conversely, let $q\in L$ be $\phi$-$\delta$-primary for some $\phi\leqslant\phi_2$. Then by Lemma \ref{L-C1}, $q\in L$ is $\phi_2$-$\delta$-primary (almost $\delta$-primary). Let $xy\leqslant q$ for $x,\ y\in L$. If $xy\nleqslant \phi_2(q)$ then as $q$ is $\phi_2$-$\delta$-primary, we have either $x\leqslant q$ or $y\leqslant\delta(q)$. If $xy\leqslant \phi_2(q)=q^2$ then as $q$ is 2-potent $\delta$-primary, we have either $x\leqslant q$ or $y\leqslant\delta(q)$. Hence $q$ is $\delta$-primary.
   \end{proof}
   
   \begin{c1}
   Every $\phi_2$-$\delta$-primary element of $L$ which is $2$-potent $\delta$-primary is $\delta$-primary.
   \end{c1}
   \begin{proof}
   The proof follows from proof of the Theorem $\ref{T-C80}$.
   \end{proof}
   
   \begin{thm}\label{T-C73}
   Let a proper element $q\in L$ be 2-potent $\delta_0$-primary. Then $q$ is $\phi$-$\delta$-primary for some $\phi\leqslant\phi_2$ if and only if $q$ is $\delta$-primary.
   \end{thm}
   \begin{proof}
   The proof follows by using Theorem $\ref{T-C80}$ to the fact that every 2-potent $\delta_0$-primary element of $L$ is $2$-potent $\delta$-primary.
   \end{proof}
   
   \begin{c1}
      Every $\phi_2$-$\delta$-primary element of $L$ which is $2$-potent $\delta_0$-primary is $\delta$-primary.
      \end{c1}
      
   \begin{d1} 
      Let $n\geqslant 2$. A proper element $p\in L$ is said to be {\bf $n$-potent $\delta$-primary} if for all $a,\ b\in L$, $ab\leqslant p^n$ implies either $a\leqslant p$ or $b\leqslant \delta(p)$.
      \end{d1} 
      
 Obviously, every $n$-potent $\delta_0$-primary element of $L$ is $n$-potent $\delta$-primary.\\    
      
 The following result is general form of Theorem \ref{T-C80}.     
      
     \begin{thm}\label{T-C3}
      A proper element  $q\in L$ is $\phi$-$\delta$-primary for some $\phi\leqslant\phi_n$ where $n\geqslant 2$ if and only if $q$ is $\delta$-primary, provided $q$ is $k$-potent $\delta$-primary for some $k\leqslant n$.
      \end{thm} 
      \begin{proof}
      Assume that $q\in L$ is a $\delta$-primary element. Then obviously, $q$ is $\phi$-$\delta$-primary for every $\phi$ and hence for some $\phi\leqslant \phi_n$ where $n\geqslant 2$. Conversely, let $q\in L$ be $\phi$-$\delta$-primary for some $\phi\leqslant\phi_n$ where $n\geqslant 2$. Then by Lemma \ref{L-C1}, $q\in L$ is $\phi_n$-$\delta$-primary ($n$-almost $\delta$-primary). Let $xy\leqslant q$ for $x,\ y\in L$. If $xy\nleqslant \phi_k(q)=q^k$ then $xy\nleqslant \phi_n(q)=q^n$ as $k\leqslant n$. Since $q$ is $\phi_n$-$\delta$-primary, we have either $x\leqslant q$ or $y\leqslant\delta(q)$. If $xy\leqslant \phi_k(q)=q^k$ then as $q$ is $k$-potent $\delta$-primary, we have either $x\leqslant q$ or $y\leqslant\delta(q)$. Hence $q$ is $\delta$-primary.
      \end{proof}  
      
   \begin{c1}
         Every $\phi_n$-$\delta$-primary element of $L$ which is $k$-potent $\delta$-primary is $\delta$-primary where $k\leqslant n$. 
         \end{c1}    
      
  \begin{thm}\label{T-C23}
   Let a proper element $q\in L$ be $\phi$-$\delta$-primary. If $q^2\nleqslant \phi(q)$ then $q$ is $\delta$-primary.
   \end{thm}
   \begin{proof}
   Let $ab\leqslant q$ for $a,\ b\in L$. If $ab\nleqslant \phi(q)$ then as $q$ is $\phi$-$\delta$-primary, we have either $a\leqslant q$ or $b\leqslant\delta(q)$. So assume that $ab\leqslant\phi(q)$. First suppose $aq\nleqslant\phi(q)$. Then $ad\nleqslant\phi(q)$ for some $d\leqslant q$ in $L$. Also $a(b\vee d)=ab\vee ad\leqslant q$ and $a(b\vee d)\nleqslant\phi(q)$. As $q$ is $\phi$-$\delta$-primary, either $a\leqslant q$ or $(b\vee d)\leqslant\delta(q)$. Hence either $a\leqslant q$ or $b\leqslant\delta(q)$. Similarly, if $bq\nleqslant \phi(q)$, we can show that either $a\leqslant q$ or $b\leqslant\delta(q)$. So we can assume that $aq\leqslant \phi(q)$ and $bq\leqslant\phi(q)$. Since $q^2\nleqslant \phi(q)$, there exist $r,\ s\leqslant q$ in $L$ such that $rs\nleqslant\phi(q)$. Then $(a\vee r)(b\vee s)\leqslant q$ but $(a\vee r)(b\vee s)\nleqslant\phi(q)$. As $q$ is $\phi$-$\delta$-primary, we have either $(a\vee r)\leqslant q$ or $(b\vee s)\leqslant\delta(q)$. Therefore either $a\leqslant q$ or $b\leqslant\delta(q)$ and hence $q$ is $\delta$-primary.
   \end{proof}
   
   From the Theorem $\ref{T-C23}$, it follows that,
   \begin{itemize}
   \item if a proper element $q\in L$ is $\phi$-$\delta$-primary but not $\delta$-primary then $q^2\leqslant \phi(q)$,
   \item a $\phi$-$\delta$-primary element $q<1$ of $L$ with $q^2\nleqslant \phi(q)$ is $\delta$-primary.
   \end{itemize}
   
     Clearly, given an expansion function $\delta$ on $L$, 
          $\delta(p)\leqslant \delta(\delta(p))$ for each $p\in L$. Moreover, for each $p\in L$, $\delta_1(\delta_1(p))=\delta_1(p)$, by property $(p3)$ of radicals in \cite{TM}. Also, obviously $\delta_0(\delta_0(p))=\delta_0(p)$ for each $p\in L$.\\
             
   Now we present the consequences of the Theorem $\ref{T-C23}$ in the form of following corollaries.
   \begin{c1}
   If a proper element $q\in L$ is $\phi$-$\delta$-primary but not $\delta$-primary then $\delta_1(q)=\delta_1(\phi(q))$.
   \end{c1}
   \begin{proof}
   By Theorem $\ref{T-C23}$, we have $q^2\leqslant \phi(q)$. So $q\leqslant\delta_1(\phi(q))$ which gives $\delta_1(q)\leqslant\delta_1(\delta_1(\phi(q)))=\delta_1(\phi(q))$. Since $\phi(q)\leqslant q$, we have $\delta_1(\phi(q))\leqslant\delta_1(q)$. Hence $\delta_1(q)=\delta_1(\phi(q))$.
   \end{proof}
   
   \begin{c1}\label{C-C1}
   If a proper element $q\in L$ is $\phi$-$\delta$-primary where $\phi\leqslant\phi_3$ then $q$ is $\phi_n$-$\delta$-primary for every $n\geqslant 2$. 
   \end{c1}
   \begin{proof}
   If $q$ is $\delta$-primary then by Theorem $\ref{T-C22}$, $q$ is $\phi_\omega$-$\delta$-primary. So assume that $q$ is not $\delta$-primary. Then by Theorem $\ref{T-C23}$ and by hypothesis, we get $q^2\leqslant\phi(q)\leqslant q^3$. Hence $\phi(q)=q^n$ for every $n\geqslant 2$. Consequently, $q$ is $\phi_n$-$\delta$-primary for every $n\geqslant 2$.  
   \end{proof}
   
   \begin{c1}
    If a proper element $q\in L$ is $\phi$-$\delta$-primary where $\phi\leqslant\phi_3$ then $q$ is $\phi_\omega$-$\delta$-primary.
   \end{c1}
   \begin{proof}
   The proof follows from Corollary $\ref{C-C1}$ and Corollary $\ref{L-C21}$.
   \end{proof}
   
   \begin{c1}
   If a proper element $q\in L$ is $\phi_0$-$\delta$-primary but not $\delta$-primary then $q^2=0$.
   \end{c1}
   \begin{proof}
   The proof is obvious.
   \end{proof}
   
   \begin{thm}
   Let $q$ be a $\phi$-$\delta$-primary element of $L$. If $\phi(q)$ is a $\delta$-primary element of $L$ then $q$ is $\delta$-primary.
   \end{thm}
   \begin{proof}
   Let $ab\leqslant q$ for $a,\ b\in L$. If $ab\nleqslant \phi(q)$ then as $q$ is $\phi$-$\delta$-primary, we have either $a\leqslant q$ or $b\leqslant \delta(q)$ and we are done. Now if $ab\leqslant \phi(q)$ then $\phi(q)$ is $\delta$-primary, we have either $a\leqslant \phi(q)$ or $b\leqslant \delta(\phi(q))$. This implies that either $a\leqslant q$ or $b\leqslant \delta(q)$ because  $\phi(q)\leqslant q$ and $\delta(\phi(q))\leqslant \delta(q)$. 
   \end{proof}
   
 The next result shows that the join of a family of ascending chain of $\phi$-$\delta$-primary elements of $L$ is again  $\phi$-$\delta$-primary.
       
       \begin{thm}
       Let $\{p_i\mid i\in \triangle\}$ be a chain of $\phi$-$\delta$-primary elements of $L$ and let the function $\phi$ be such that $x\leqslant y$ imply $\phi(x)\leqslant \phi(y)$ for all $x, y\in L$. Then the element $p=\underset{i\in \triangle}{\vee}p_i$ is also $\phi$-$\delta$-primary.
       \end{thm}
       \begin{proof}
       Since $1\in L$ is compact, $\underset{i\in \triangle}{\vee}p_i=p\neq 1$. Let $ab\leqslant p$, $ab\nleqslant \phi(p)$ and $a\nleqslant p$ for $a,\ b\in L$. Then as $\{p_i\mid i\in \triangle\}$ is a chain, we have $ab\leqslant p_i$ for some $i\in \triangle$ but $a\nleqslant p_i$ and $ab\nleqslant \phi(p_i)$ because for each $k\in \triangle$, we have $p_k\leqslant p$ and this implies $\phi(p_k)\leqslant \phi(p)$. As each $p_i$ is $\phi$-$\delta$-primary, it follows that $b\leqslant \delta(p_i)$. Since $p_i\leqslant p$, we have $\delta( p_i)\leqslant\delta(p)$ and so $b\leqslant\delta(p)$. Hence $p$ is $\phi$-$\delta$-primary.
       \end{proof}
       
    The following theorem shows that a under certain condition, $(p:q)\in L$ is $\phi$-$\delta$-primary if $p\in L$ is   $\phi$-$\delta$-primary element where $q\in L$.
       
       \begin{thm}
       Let a proper element $p\in L$ be $\phi$-$\delta$-primary. Then  $(p:q)$ is $\phi$-$\delta$-primary for all $q\in L$ if $(\phi(p):q)\leqslant \phi(p:q)$.
       \end{thm}       
       \begin{proof}
       Clearly, $pq\leqslant p$ implies $p\leqslant (p:q)$ and so $\delta(p)\leqslant\delta(p:q)$. Now let $ab\leqslant (p:q)$, $ab\nleqslant \phi(p:q)$ and $a\nleqslant (p:q)$ for $a,\ b\in L$. Then $abq\leqslant p$, $abq\nleqslant \phi(p)$ and $aq\nleqslant p$ since $ab\nleqslant (\phi(p):q)$. Now as $p$ is $\phi$-$\delta$-primary, we have $b\leqslant \delta(p)\leqslant\delta(p:q)$ and hence $(p:q)$ is $\phi$-$\delta$-primary. 
       \end{proof} 
       
       In the next result, we show that under a certain condition $\delta_1(p)\leqslant \delta(p)$, for every $\phi$-$\delta$-primary $p\in L$.
       
       \begin{thm}\label{TL3}
        If a proper element $p\in L$ is $\phi$-$\delta$-primary element such that $\delta_1(\phi(p)) \leqslant \delta(p)$ then  $\delta_1(p)\leqslant \delta(p)$. 
       \end{thm}
       \begin{proof}
       Assume that a proper element $p\in L$ is $\phi$-$\delta$-primary. For $a\in L$, let $a\leqslant\delta_1(p)=\sqrt{p}$. Then there exists a least positive integer $k$ such that $a^{k}\leqslant p$. If $k=1$ then $a\leqslant p\leqslant \delta(p)$. Now let $k>1$. If $a^k \leqslant \phi(p)$ then $a\leqslant \delta_1(\phi(p)) \leqslant \delta(p)$. So let $a^k \nleqslant \phi(p)$.  Clearly, $a^{k-1}a\leqslant p$ where $a^{k-1}\nleqslant p$.  As $p\in L$ is $\phi$-$\delta$-primary, it follows that $a\leqslant\delta(p)$. Thus in any case, we have $\delta_1(p)\leqslant\delta(p)$.
       \end{proof}
       
       Note that, if $p\in L$ is $\delta$-primary then by consequence of Theorem 2.5 of \cite{MB1}, we have $\phi(p)\leqslant p$ implies $\delta_1(\phi(p))\leqslant \delta_1(p)\leqslant \delta(p)$ and hence $\delta_1(\phi(p)) \leqslant \delta(p)$.
       
       \begin{c1}
       If a proper element $p\in L$ is $\phi$-$\delta$-primary element such that $\delta_1(\phi(p)) \leqslant \delta(p)$ with $\delta(p)\leqslant \delta_1(p)$ then $\delta_1(p)= \delta(p)$. 
       \end{c1}
       \begin{proof}
       The proof follows from Theorem \ref{TL3}.
       \end{proof}
       
       According to \cite{MB1}, an expansion function $\delta$ on $L_1$ and on $L_2$ is said to have global property if for any lattice isomorphism $f:L_1\longrightarrow L_2$,  $\delta(f^{-1}(a))=f^{-1}(\delta(a))$ for all $a\in L_2$ where $L_1$ and $L_2$ are multiplicative lattices.  Similarly, now we define global property of a function $\phi$ on multiplicative lattices.
       \begin{d1} 
       Let $L_1$ and $L_2$ be multiplicative lattices. A  function $\phi$ on $L_1$ and on $L_2$ is said to have {\bf global property} if for any lattice isomorphism $f:L_1\longrightarrow L_2$,  $\phi(f^{-1}(a))=f^{-1}(\phi(a))$ for all $a\in L_2$. 
       \end{d1}
       
       \begin{l1}\label{L-C2}
       Let the function $\beta$ on $L_1$ and on $L_2$ have the global property where $L_1$ and $L_2$ are multiplicative lattices.  If the function $g:L_1\longrightarrow L_2$ is a lattice isomorphism then $g(\beta(q))=\beta(g(q))$ for all $q\in L_1$.
       \end{l1}
       \begin{proof}
       For $q\in L_1$, the global property of $\beta$ gives $\beta(q)= \beta(g^{-1}(g(q)))=g^{-1}(\beta(g(q)))$. Then since $g$ is onto, we have $g(\beta(q))=\beta(g(q))$.
       \end{proof}
       
The next result shows that if $q\in L$ is $\phi$-$\delta$-primary with some conditions on $\delta$ and $\phi$ then $\delta(q)\in L$ is $\phi$-prime element.
       
       \begin{thm}\label{T-C70}
                     Let the expansion function $\delta$ on $L$ be a lattice isomorphism. Let the function $\phi$ on $L$ have the global property. If a proper element $q\in L$ is $\phi$-$\delta$-primary and satisfies $\delta(\delta(q))\leqslant\delta(q)$ then $\delta(q)$ is a $\phi$-prime element of $L$.
                     \end{thm}
                     \begin{proof}
                     By Lemma \ref{L-C2}, we have $\delta(\phi(q))=\phi(\delta(q))$.
                     Let $xy\leqslant \delta(q)$, $xy\nleqslant \phi(\delta(q))=\delta(\phi(q))$ and $x\nleqslant \delta(q)$ for $x,\ y\in L$.  So  $\delta^{-1}(x)\cdot \delta^{-1}(y)=\delta^{-1}(xy)\leqslant q$, $\delta^{-1}(x)\cdot \delta^{-1}(y)=\delta^{-1}(xy)\nleqslant \phi(q)$  and $\delta^{-1}(x)\leqslant q$. As $q$ is $\phi$-$\delta$-primary, we have $\delta^{-1}(y)\leqslant \delta(q)$ which implies $y\leqslant \delta(\delta(q))\leqslant\delta(q)$ and hence $\delta(q)$ is a $\phi$-prime element of $L$. 
                     \end{proof}
                     
                     Note that, in the Theorem \ref{T-C70}, the idea behind taking the expansion function $\delta$ on $L$ as a lattice isomorphism and the function $\phi$ on $L$ with the global property is to get $\delta(\phi(q))=\phi(\delta(q))$.  The following theorem is a similar version of Theorem \ref{T-C70}.
                     
                     \begin{thm}
                     If a proper element $q\in L$ is $\phi$-$\delta_1$-primary such that $\delta_1(\phi(q))=\phi(\delta_1(q))$ then $\delta_1(q)$ is a $\phi$-prime element of $L$.
                     \end{thm}
                     \begin{proof} Assume that $ab\leqslant \delta_1(q)$, $ab\nleqslant\phi(\delta_1(q))$ and $a\nleqslant \delta_1(q)$ for $a,\ b\in L$. Then there exists $n\in Z_+$ such that $a^n\cdot b^n=(ab)^n\leqslant q$. If $(ab)^n\leqslant\phi(q)$ then by hypothesis $ab\leqslant\delta_1(\phi(q))=\phi(\delta_1(q))$, a contradiction. So we must have $a^n\cdot b^n=(ab)^n\nleqslant\phi(q)$. Since $q$ is $\phi$-$\delta_1$-primary and $a^n\nleqslant q$ for all $n\in Z_+$, we have $b^n\leqslant\delta_1(q)$ and hence $b\leqslant\delta_1(\delta_1(q))=\delta_1(q)$. This shows that $\delta_1(q)$ is a $\phi$-prime element of $L$.
                     \end{proof}      
              
     \begin{l1}\label{L-1}
       Let the expansion function $\delta$ on $L_1$ and on $L_2$ have the global property where $L_1$ and $L_2$ are multiplicative lattices. Let the function $\phi$ on $L_1$ and on $L_2$ have the global property. If $f:L_1\longrightarrow L_2$ is a lattice isomorphism then for any $\phi$-$\delta$-primary element $p\in L_2$,  $f^{-1}(p)\in L_1$ is $\phi$-$\delta$-primary.
       \end{l1} 
       \begin{proof}
       Assume that a proper element $p\in L_2$ is $\phi$-$\delta$-primary. Let $ab\leqslant f^{-1}(p)$, $ab\nleqslant \phi(f^{-1}(p))=f^{-1}(\phi(p))$ and $a\nleqslant f^{-1}(p)$ for $a,\ b\in L_1$. Then $f(ab)=f(a)\cdot f(b)\leqslant p$, $f(ab)=f(a)\cdot f(b)\nleqslant \phi(p)$  and $f(a)\nleqslant p$. As $p$ is $\phi$-$\delta$-primary, we have $f(b)\leqslant \delta(p)$. Now the global property of $\delta$ gives $b\leqslant f^{-1}(\delta(p))=\delta(f^{-1}(p))$ showing that $f^{-1}(p)\in L_1$ is $\phi$-$\delta$-primary.
       \end{proof}       
        
       The following result gives another characterization of $\phi$-$\delta$-primary elements of $L$.
       
       \begin{thm}
         Let the expansion function $\delta$ on $L_1$ and on $L_2$ have the global property where $L_1$ and $L_2$ are multiplicative lattices. Let the function $\phi$ on $L_1$ and on $L_2$ have the global property. Let $f:L_1\longrightarrow L_2$ be a lattice isomorphism. Then a proper element $a\in L_1$ is $\phi$-$\delta$-primary if and only if $f(a)\in L_2$ is $\phi$-$\delta$-primary.
       \end{thm} 
       \begin{proof}
       Assume that a proper element $a\in L_1$ is $\phi$-$\delta$-primary. Clearly, by Lemma \ref{L-C2}, the global property of $\delta$ gives $f(\delta(a))=\delta(f(a))$. Also, by Lemma \ref{L-C2}, the global property of $\phi$ gives $f(\phi(a))=\phi(f(a))$. Now, let $xy\leqslant f(a)$, $xy\nleqslant \phi(f(a))$ and $x\nleqslant f(a)$ for $x,\ y\in L_2$. Then there exists $b,\ c\in L_1$ such that $f(b)=x,\ f(c)=y$. So $f(bc)=f(b)\cdot f(c)=xy\leqslant f(a)$, $f(bc)=f(b)\cdot f(c)=xy\nleqslant \phi(f(a))=f(\phi(a))$ and $f(b)=x\nleqslant f(a)$. As $a$ is $\phi$-$\delta$-primary in $L_1$, $bc\leqslant a$, $bc\nleqslant \phi(a)$ and $b\nleqslant a$, we have $c\leqslant \delta(a)$. So $y=f(c)\leqslant f(\delta(a))$ and hence $y\leqslant \delta(f(a))$ showing that $f(a)\in L_2$ is $\phi$-$\delta$-primary. The converse follows from Lemma \ref{L-1}. 
       \end{proof} 
       
       Now we relate idempotent element of $L$ with $\phi_n$-$\delta$-primary element $(n\geqslant 2)$ of $L$.
       
       \begin{thm}\label{T-C71}
       Every idempotent element of $L$ is $\phi_\omega$-$\delta$-primary and hence $\phi_n$-$\delta$-primary $(n\geqslant 2)$.
       \end{thm}
       \begin{proof}
       Let $p$ be an idempotent element of $L$. Then $p=p^n$ for all $n\in Z_+$. So $\phi_\omega(p)=p$. Therefore $p$ is a $\phi_\omega$-$\delta$-primary of $L$. Hence $p$ is a $\phi_n$-$\delta$-primary element $(n\geqslant 2)$ of $L$ by Theorem \ref{T-C22}.
       \end{proof}
       
       As a consequence of Theorem \ref{T-C71}, we have following result whose proof is obvious.
       
       \begin{c1}
       Every idempotent element of $L$ is $\phi_2$-$\delta$-primary.
       \end{c1}
       
       However, a $\phi_2$-$\delta$-primary element of $L$ need not be idempotent as shown in the following example (by taking $\delta$ as $\delta_1$ for convenience).
       
       \begin{e1}\label{E-C23}
       Consider the lattice $L$ of ideals of the ring $R=<Z_8\ , \ +\ ,\ \cdot>$. Then the only ideals of $R$ are the principal ideals (0),(2),(4),(1). Clearly, $L=\{(0),(2),(4),(1)\}$ is a compactly generated multiplicative lattice. Its lattice structure and multiplication table is as shown in Figure 2.3.  It is easy to see that the element $(4)\in L$ is $\phi_2$-$\delta_1$-primary but not idempotent.
       
       \centerline{\includegraphics[width = 125mm, height = 75mm]{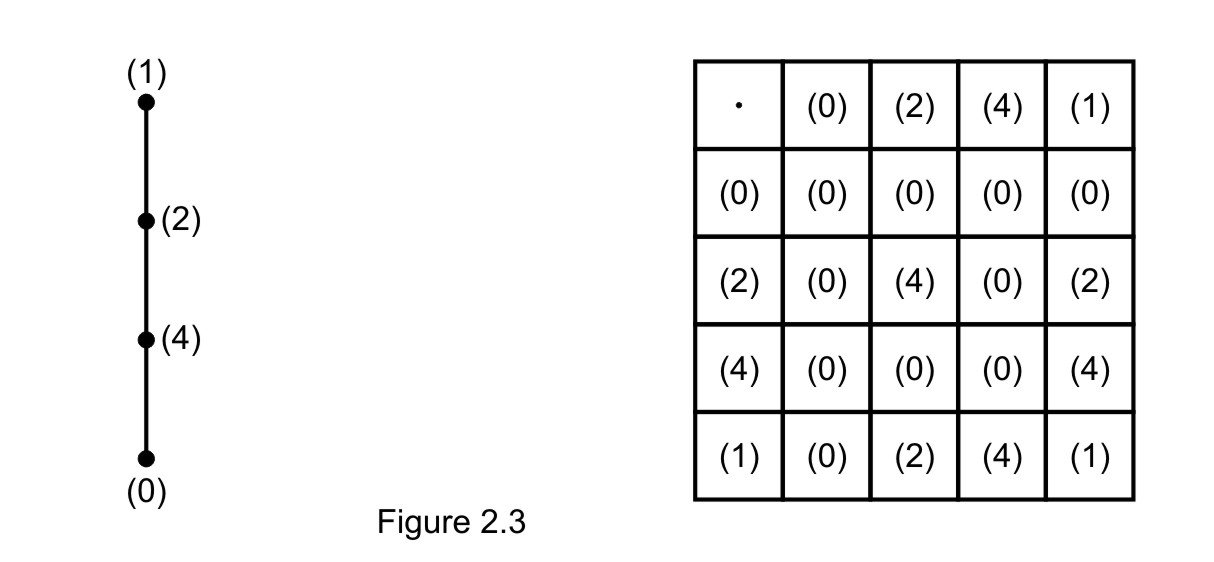}}    
     
       \end{e1}
       
       We conclude this paper with the following examples, from which it is clear that,
       
       \textcircled{1} a $\phi_2$-$\delta_1$-primary element of $L$ need not be 2-potent $\delta_0$-primary,
        
       \textcircled{2} a 2-potent $\delta_0$-primary element of $L$ which is $\phi_2$-$\delta_1$-primary need not be prime.
       
       \begin{e1}
       Consider $L$ as in Example $\ref{E-C22}$. Here the element $(6)\in L$ is $\phi_2$-$\delta_1$-primary but not $2$-potent $\delta_0$-primary.
       \end{e1}
       
       \begin{e1}
       Consider $L$ as in Example $\ref{E-C23}$. Here the element $(4)\in L$ is $2$-potent $\delta_0$-primary, $\phi_2$-$\delta_1$-primary but not prime.
       \end{e1}


\begin{thebibliography}{99}

\bibitem{AAJ}
       {F.~Alarcon, D.~D.~Anderson and C.~Jayaram},
       {\it Some results on abstract commutative ideal theory}, Periodica Mathematica Hungarica, {\bf 30(1)}(1995), 1-26.
       
       \bibitem{AB2}
       {D.~D.~Anderson and M.~Bataineh}, {\it Generalizations of prime ideals},  Communications in Algebra, {\bf 36(2)}, (2008), 686-696.
      
       \bibitem{D}
       { A.~Y.~Darani}, {\it Generalizations of primary ideals in commutative rings}, Novi Sad J. Math., {\bf 42(1)}, (2012), 27-35.
       
      \bibitem{D3}
       {R.~P.~Dilworth}, {\it Abstract commutative ideal theory}, Pacific Journal of Mathematics, {\bf 12(2)}, (1962), 481-498.
      
      \bibitem{EN}
       {M.~Ebrahimpour and R.~Nekooei}, {\it On generalizations of prime ideals}, Communications in Algebra, {\bf 40 (4)}, (2012), 1268-1279.
             
       \bibitem{EN2}
       {M.~Ebrahimpour and R.~Nekooei}, {\it On generalizations of prime submodules}, Bulletin of the Iranian Mathematical Society, {\bf 39(5)}, (2013), 919-939.
       
            \bibitem{J1}
             {A.~Jaber},
             {\it Properties of $\phi$-$\delta$-primary and 2-absorbing
             $\delta$-primary ideals of commutative rings}, Asian-European Journal of Mathematics, {\bf 13(1)}(2020), 1-11.
            
       \bibitem{MB1}
       {C.~S.~Manjarekar and A.~V.~Bingi},
       {\it $\delta$-primary elements in multiplicative lattices}, International Journal of Advance Research, {\bf 2(6)} (2014), 1-7.
                          
        \bibitem{MB}
        {C.~S.~Manjarekar and A.~V.~Bingi},
        {\it $\phi$-Prime and $\phi$-Primary Elements in Multiplicative Lattices}, Algebra, (2014), 1-7.
    
       \bibitem{S0}
       {D.~S.~Culhan}, {\it Associated primes and primal decomposition in modules and lattice modules and their duals}, Ph.~D. Thesis, University of California, Riverside, 2005.
      
      \bibitem{TM}
       {N.~K.~Thakare and C.~S.~Manjarekar},
       {\it Radicals and uniqueness theorem in multiplicative lattices with chain conditions}, Studia Scientifica Mathematicarum Hungarica, {\bf 18} (1983), 13-19.
       
        \bibitem{W}
        {J.~Wells}, {\it The restricted cancellation law in a Noether lattice}, Fundamenta Mathematicae, {\bf 3(75)}, (1972), 235-247.
       
        \bibitem{Z}
              {D.~Zhao},
              {\it $\delta$-primary ideals of commutative rings}, Kyungpook Math. J, {\bf 41}(2001), 17-22.  
       
    \end{thebibliography}
\end{document}